\documentclass[oupthm]{CUP-JNL-BCM}%

\usepackage{graphicx}
\usepackage{multicol,multirow}
\usepackage{amsmath,amssymb,amsfonts}
\usepackage{mathrsfs}
\usepackage{rotating}
\usepackage{appendix}
\usepackage[numbers,sort&compress]{natbib}
\usepackage[T5,T1]{fontenc}

\theoremstyle{oupplain}
\newtheorem{thm}{Theorem}[section]
\newtheorem{lemma}[thm]{Lemma}
\newtheorem{cor}[thm]{Corollary}
\newtheorem{prop}[thm]{Proposition}

\theoremstyle{oupdefinition}

\theoremstyle{oupremark}
\newtheorem{rmk}[thm]{Remark}

\newtheorem{ex}[thm]{Example}
\theoremstyle{oupproof}
\newtheorem{proof}{Proof}

\numberwithin{equation}{section}

\articletype{RESEARCH ARTICLE}
\jname{Canadian Mathematical Society}
\jyear{2026}

\newcommand{\Z}{\mathbb{Z}}
\newcommand{\D}{\mathbb{D}}
\renewcommand{\L}{\mathbb{L}}
\newcommand{\id}{\mathrm{id}}
\newcommand{\op}{^\mathrm{op}}
\newcommand{\inv}{^{-1}}
\newcommand{\pinv}{^{\pm 1}}
\newcommand{\avg}{_{\mathrm{mid}}}
\DeclareMathOperator{\Alex}{Alex}
\DeclareMathOperator{\FAM}{FAM}
\DeclareMathOperator{\FML}{FML}
\DeclareMathOperator{\FMC}{FMC}
\DeclareMathOperator{\Aut}{Aut}
\renewcommand{\phi}{\varphi}
\newcommand{\bij}{\xrightarrow{\sim}}
\newcommand{\surj}{\twoheadrightarrow}
\newcommand{\inj}{\hookrightarrow}
\newcommand{\AM}{\mathsf{AffMod}}
\newcommand{\Med}{\mathsf{MLQnd}}
\newcommand{\Comm}{\mathsf{MCQnd}}

\begin{document}

\begin{Frontmatter}

\title[On Medial Latin Quandles and Affine Modules]{On medial Latin quandles and affine modules}

\author{L{\fontencoding{T5}\selectfont\d\uhorn}c Ta}

\authormark{L. Ta}

\address{\orgname{Department of Mathematics, University of Pittsburgh}, \orgaddress{\street{Pittsburgh}, \state{PA}, \postcode{15260}}\email{ldt37@pitt.edu}}



\abstract{In this note, we show that the category of Latin (resp.\ commutative) medial quandles is equivalent to the category of affine modules over a certain Laurent polynomial ring (resp.\ the dyadic rationals). As applications, we describe free objects in these categories and obtain a structure theorem for finitely generated medial commutative quandles. We also characterize racks whose duals are commutative. Collectively, this solves two open problems of Bardakov and Elhamdadi \cite{BE}.}

\end{Frontmatter}

\section{Introduction}
In this note, we show that the categories of medial Latin quandles $\Med$ and medial commutative quandles $\Comm$ are equivalent to the categories of affine modules $\AM_\L$ and $\AM_\D$ over the rings
\[
\L\coloneq \Z[t\pinv, (1-t)\inv]\qquad \text{and } \qquad\D\coloneq\Z[1/2],
\]
respectively; see Theorem \ref{thm:main} (resp.\ Corollary \ref{cor:main}). As applications, we describe free objects in these categories (see Proposition \ref{prop:fml} and its preceding discussion) and prove a structure theorem for finitely generated commutative medial quandles (Theorem \ref{thm:structure}). Together with Corollaries \ref{cor:cocomm} and \ref{cor:free-comm}, this solves two problems of Bardakov and Elhamdadi \cite{BE} from the theory of quandle rings.

\subsection{Motivations}
In 1982, Joyce \cite{joyce} and Matveev \cite{matveev} independently introduced algebraic structures called \emph{quandles} to construct knot invariants. Quandles generalize \emph{kei}, which Takasaki \cite{takasaki} introduced in 1943 as abstractions of symmetric spaces. Quandles form a special subvariety of \emph{racks}, which Fenn and Rourke \cite{fenn} introduced in 1992. See \cite{book,quandlebook} for introductions to the theory.

\emph{Medial} quandles, a special class of quandles of significant interest in the literature, enjoy rich algebraic properties (see \cite{jedlicka,ta2}) and connections to Alexander invariants of knots (see \cite{joyce,matveev}). On the other hand, \emph{Latin} quandles and \emph{commutative} quandles are important to the theory of \emph{quandle rings}, which are nonassociative analogues of group rings used to enhance coloring invariants of knots; see \cite{BE} for references. In particular, medial commutative quandles form an admissible subvariety $\Comm$ for the categorical theory of central extensions (see \cite{even}), making it crucial to better understand the structure of $\Comm$ as a category. Sangare \cite{sangare1,sangare2} recently used the results of the present paper to study automorphisms and antiautomorphisms of medial commutative quandles.

This note offers a categorical perspective on the correspondence between medial Latin quandles (resp.\ medial commutative quandles) and affine modules over the Laurent polynomial ring $\L$ (resp.\ the dyadic rationals $\D$), which was previously studied in \cite{jedlicka,bauer,kepka}. As an application, we solve two problems of Bardakov and Elhamdadi \cite[][Questions 7.1 and 7.3]{BE} from the theory of quandle rings.

\subsection{Structure of the paper}

In Section \ref{sec:2}, we define several classes of quandles, including \emph{Alexander quandles} $\Alex(A,\phi)$ and \emph{midpoint quandles} $M\avg$, and we define the category of affine modules $\AM_R$ over a ring $R$.

In Section \ref{sec:3}, we solve \cite[][Question 7.1]{BE}. In particular, we discuss commutative quandles that are not medial, and we completely characterize racks (see Theorem \ref{thm:cocomm}) and commutative quandles (see Corollary \ref{cor:cocomm}) $X$ whose duals $X\op$ are commutative.

In Section \ref{sec:4}, we provide short new proofs that all medial Latin quandles (resp.\ medial commutative quandles) are isomorphic to certain Alexander quandles (resp.\ midpoint quandles); see Proposition \ref{prop:every-ML} (resp.\ Corollary \ref{cor:every-MC}).

In Section \ref{sec:5}, we categorify the results of Section \ref{sec:4} by showing that ``taking the Alexander quandle'' (resp.\ the midpoint quandle) defines an equivalence of categories $\Alex\colon \AM_\L\bij\Med$ (resp.\ $\operatorname{mid}\colon \AM_\D\bij\Comm$); see Theorem \ref{thm:main} (resp.\ Corollary \ref{cor:main}).

In Section \ref{sec:6}, we characterize all free objects in the categories studied in Section \ref{sec:5} (see Proposition \ref{prop:fml} and its preceding discussion) and certain free objects in the category of commutative quandles (see Corollary \ref{cor:free-comm}). We deduce a structure theorem (Theorem \ref{thm:structure}) for finitely generated medial commutative racks. This solves Question 7.3 and resolves a conjecture in Question 7.1 of \cite{BE}.

Throughout this note, all abelian groups are notated additively with identity element $0$.

\subsection*{Acknowledgments} This note is adapted from \cite{ta}. I thank Valeriy Bardakov and Mohamed Elhamdadi for helpful comments and their suggestion to expand \cite{ta} into the current paper.

\section{Preliminaries}\label{sec:2}

\subsection{Quandles}
Recall that a \emph{magma} is a pair $(X,\ast)$ where $X$ is a set and $\ast\colon X^2\to X$ is a binary operation. If $(X,\ast)$ and $(Y,\cdot)$ are magmas, then a \emph{homomorphism} is a function $f\colon X\to Y$ such that $f(w\ast x)=f(w)\cdot f(x)$ for all $w,x\in X$. For all $x\in X$, define the corresponding \emph{right multiplication map} $R_x\colon X\to X$ by $y\mapsto y\ast x$, and define the \emph{left multiplication map} $L_x\colon X\to X$ by $y\mapsto x\ast y$.

In the following, let $(X,\ast)$ be a magma. We consider several important varieties of magmas.
\begin{itemize}
    \item We say that $(X,\ast)$ is \emph{idempotent} if $x\ast x=x$ for all $x\in X$.
    \item We say that $(X,\ast)$ is a \emph{rack} if, for all $x\in X$, the right-multiplication map $R_x$ is an automorphism of $X$. In particular, $\ast$ distributes over itself on the right. 
    \item \emph{Quandles} are idempotent racks. \emph{Kei} are quandles such that, for all $x\in X$, the right multiplication map $R_x$ is an involution.
    \item Equip $X^2$ with the product magma structure. We say that $(X,\ast)$ is \emph{medial} if $\ast\colon X^2\to X$ is a magma homomorphism. Several authors prefer the terms \emph{abelian} or \emph{entropic}.
    \item We say that $(X,\ast)$ is \emph{Latin} or a \emph{left quasigroup} if, for all $x\in X$, the left multiplication map $L_x$ is a permutation.
    \item We say that $(X,\ast)$ is \emph{commutative} or \emph{symmetric} if $x\ast y=y\ast x$ for all $x,y\in X$; that is, $L_x=R_x$ for all $x\in X$.
\end{itemize}

The following is well-known and straightforward.

\begin{lemma}\label{lem:comm-latin}
    All Latin racks are quandles. Moreover, all commutative racks are Latin quandles.
\end{lemma}

\subsubsection{}
Two auxiliary results of this paper (Proposition \ref{prop:every-ML} and Corollary \ref{cor:every-MC}) state that all medial Latin quandles and medial commutative quandles fall into the following two classes of examples.

\begin{ex}[\cite{nelson,jedlicka}]\label{ex:alex}
    \emph{Alexander quandles,} also called \emph{affine quandles,} form an important class of medial quandles. Given an automorphism $\phi$ of an abelian group $A$, define the Alexander quandle $\Alex(A,\phi)\coloneq(A,\ast)$ by
\begin{equation}\label{eq:alex}
    x\ast y\coloneq \phi(x)+(\id-\phi)(y).
\end{equation}
Note that $\Alex(A,\phi)$ is Latin if and only if $\id-\phi$ is an automorphism of $A$; cf.\ \cite{jedlicka}. Actually, every medial Latin quandle is isomorphic to an Alexander quandle of this form; see Proposition \ref{prop:every-ML}.
\end{ex}

\begin{ex}
    Our prototypical examples of commutative quandles generalize the commutative quandles constructed in \cite[][Ex.\ 5.1]{BE}; we call these \emph{midpoint quandles}.
    Let $R$ be a unital ring in which $2$ is invertible, and let $M$ be a left $R$-module. (Note that $M$ can also be viewed as a module over $\D$.) We define the midpoint quandle $M\avg\coloneq (M,\ast)$ by
    \[
    x\ast y\coloneq \frac{1}{2}(x+y),
    \]
    which is commutative. 
    In particular, let $n\geq 0$, let $R\coloneq \D$, and let $M\coloneq\Z/(2n+1)$ be the cyclic group of order $2n+1$. Then we call $C_{2n+1}\coloneq M\avg$ the \emph{cyclic midpoint quandle} of order $2n+1$. Note that this definition of $C_{2n+1}$ coincides with the one in \cite{BE} because $2\inv=n+1$ in $M$.
\end{ex}

\begin{rmk}\label{rmk:medial}
    Midpoint quandles are a special class of so-called \emph{weighted average quandles} (see \cite{burrows}), which are themselves a special class of Alexander quandles. Indeed, if $M\avg$ is a midpoint quandle, then multiplication by $1/2$ is an $R$-module automorphism of $M$ whose corresponding Alexander quandle is precisely $M\avg$. The converse also holds. In particular, all midpoint quandles are medial. 
\end{rmk}

\subsection{Affine modules}
Given a ring $R$, we define the category $\AM_R$ of \emph{affine modules} over $R$ (also called \emph{affine $R$-modules}) as follows. The objects of $\AM_R$ are left $R$-modules along with the empty set.\footnote{We include the empty set in $\AM_R$ because the categories in Remark \ref{rmk:lawvere}, Theorem \ref{thm:main}, and Corollary \ref{cor:main} also include the empty set as an object.} The morphisms are sums of $R$-module homomorphisms and constant maps, which we call \emph{affine transformations.}

The following is straightforward.
\begin{lemma}\label{lemma:ess-inj}
    The inclusion functor from the category of left $R$-modules into $\AM_R$ is essentially injective.
\end{lemma}
In other words, two nonempty affine $R$-modules are isomorphic if and only if they are isomorphic as left $R$-modules in the usual sense.

\subsubsection{}
Given a ring homomorphism $R\to S$, we can define the \emph{restriction of scalars} functor $\operatorname{res}\colon\AM_S\to\AM_R$ in exactly the same way as in ordinary module theory. In particular, we make the following observation.

\begin{lemma}\label{lemma:res}
    Let $R\surj S$ be a surjective ring homomorphism. Then the restriction of scalars functor $\operatorname{res}\colon\AM_S\inj\AM_R$ is fully faithful.
\end{lemma}

\begin{rmk}\label{rmk:lawvere}
    Our definition of $\AM_R$ is one of many equivalent ways to define the category of affine $R$-modules. The usual definition (as in \cite{affine,lawvere}) is as the category $\mathcal{A}$ of set-theoretic models of the Lawvere theory whose $n$-ary operations are precisely the $n$-tuples $(r_1,\dots,r_n)\in R^n$ such that $\sum^n_{i=1}r_i=1$; such an operation represents linear combinations $(x_1,\dots,x_n)\mapsto\sum^n_{i=1}r_ix_i$. In particular, the empty set is an object in $\mathcal{A}$ due to the lack of a nullary operation.

    For fun, we encourage the reader to verify that the forgetful functor $\AM_R\to \mathcal{A}$ is an equivalence of categories. As a hint, essential surjectivity is shown as follows. Given a nonempty object $M\in\mathcal{C}$, fix a basepoint $0\in M$. Define the following left $R$-module structure $(\oplus,\cdot)$ on $M$, where $0$ plays the role of the identity element of $(M,\oplus)$ as an abelian group:
    \[
    x\oplus y\coloneq 1x+1y-1(0), \qquad r\cdot x\coloneq rx+(1-r)0.
    \]
\end{rmk}

\section{Commutative quandles}\label{sec:3}
In this section, we solve \cite[][Question 7.1]{BE}. 

\subsection{Non-mediality}
The first part of \cite[][Question 7.1]{BE} asks whether every finite commutative quandle can be written as a direct sum of cyclic midpoint quandles $C_{2n+1}$. This holds for all finite medial commutative quandles (see Theorem \ref{thm:structure}) but not for all finite commutative quandles.

\subsubsection{}
To give a negative answer to the first part of \cite[][Question 7.1]{BE}, it suffices to find a finite commutative quandle that is not medial. This is because cyclic midpoint quandles are medial (see Remark \ref{rmk:medial}), and the direct sum of medial quandles is also medial. 

Indeed, Kepka and N\v emec \cite[][Cor.\ 12.5]{moufang} constructed such two such quandles in 1981.\footnote{In \cite{moufang}, Kepka and N\v emec called these algebraic structures \emph{commutative distributive quasigroups.} One can use Lemma \ref{lem:comm-latin} to show that commutative distributive quasigroups are the same as commutative quandles; cf.\ \cite{stanovsky}.} They called these quandles $D(1)$ and $D(3)$; these quandles are the first and third entries of \cite[][Ex.\ 3.4]{stanovsky}. We refer the reader to these two articles or the author's blog post \cite[][Sec.\ 3.1]{ta} for a detailed description of these quandles and their failure to be medial. 

It is worth noting that $D(1)$ and $D(3)$ are quandles of order 81, and no other non-medial commutative quandles of order up to 81 exist \cite[][Cor.\ 12.5]{moufang}. Therefore, Theorem \ref{thm:structure} provides a positive answer to the first part of \cite[][Question 7.1]{BE} for all commutative quandles up to order 80.

\subsection{Cocommutative quandles}
The second part of \cite[][Question 7.1]{BE} seeks a description of commutative quandles whose duals are commutative. Recall that if $(X,\ast)$ is a rack, then the \emph{dual rack} $X\op\coloneq (X,\overline{\ast})$ is the rack defined by
\[
x\mathop{\overline{\ast}}y\coloneq R_y\inv(x).
\]
See \cite{burrows} for a proof that $X\op$ is a rack. 

We will say that a rack $(X,\ast)$ is \emph{cocommutative} if $X\op$ is commutative. We completely characterize cocommutative racks and, in particular, commutative cocommutative quandles. This solves \cite[][Question 7.1]{BE}.

\begin{thm}\label{thm:cocomm}
    Let $(X,\ast)$ be a rack. The following are equivalent:
    \begin{enumerate}
        \item\label{item:cocomm} $X$ is cocommutative.
        \item\label{item:cocomm-qnd} $X$ is a cocommutative quandle.
        \item\label{item:lx} For all $x\in X$, the left multiplication map $L_x$ is an involution.
    \end{enumerate}
\end{thm}

\begin{proof}
    First, note that condition \eqref{item:lx} is equivalent to the statement that, for all $x,y\in X$,
    \begin{equation}\label{eq:lx}
        x\ast (x\ast y)=y.
    \end{equation}

    \eqref{item:cocomm}$\implies$\eqref{item:cocomm-qnd}: If $X$ is cocommutative, then Lemma \ref{lem:comm-latin} implies that $X\op$ is a quandle. That is, $R_x\inv(x)=x$ for all $x\in X$. Applying $R_x$ to both sides shows that $X$ is also a quandle.

    \eqref{item:cocomm-qnd}$\implies$\eqref{item:lx}: Suppose that $X$ is a cocommutative quandle. Given $x,y\in X$, let $z\coloneq R_y\inv(x)$, so $z=R_x\inv(y)$ by assumption. Since $X$ is a quandle, it follows that
    \[
    x\ast(x\ast y)=(z\ast y)\ast(x\ast y)=(z\ast x)\ast y=y\ast y=y,
    \]
    which verifies \eqref{eq:lx}.

    \eqref{item:lx}$\implies$\eqref{item:cocomm}: Given $x,y\in X$, let $z\coloneq R_y\inv(x)$. We have to show that $z=R_x\inv(y)$. Indeed,
    \[
    R_x(z)=R_{R_y(z)}(z)=z\ast (z\ast y)=y.
    \]
    In the last equality, we have used \eqref{eq:lx}. Since $R_x$ is invertible, the proof is complete.
\end{proof}

\begin{cor}
    Every cocommutative rack is a Latin quandle.
\end{cor}

\begin{cor}[{cf.\ \cite[][Question 7.1]{BE}}]\label{cor:cocomm}
    Let $(X,\ast)$ be a commutative quandle. Then $X$ is cocommutative if and only if $X$ is a kei.
\end{cor}

\begin{proof}
    Since $X$ is commutative, $L_x=R_x$ for all $x\in X$. Hence, Theorem \ref{thm:cocomm} provides the claim.
\end{proof}

\begin{rmk}
    Commutative kei are the same as \emph{distributive Steiner quasigroups,} which are algebraic structures that correspond to combinatorial designs called \emph{Hall triple systems}. See \cite[][Sec.\ 3.4]{involutory} for further discussion and references on distributive Steiner quasigroups, and see \cite{kei} for a more quandle-theoretic treatment of commutative kei.
\end{rmk}

\section{Classification of medial Latin quandles}\label{sec:4}

In preparation for Theorem \ref{thm:main} and Corollary \ref{cor:main}, we completely describe medial Latin quandles and medial commutative quandles. The following results were shown in \cite{jedlicka,moufang,bauer}, sometimes in more esoteric forms; we reformulate these results in terms of quandles. We also provide much shorter proofs using the classical Bruck--Murdoch--Toyoda theorem from quasigroup theory.

\subsection{}
Recall that a \emph{quasigroup} is a magma such that all multiplication maps $L_x$ and $R_x$ are permutations. In particular, all Latin quandles and (by Lemma \ref{lem:comm-latin}) commutative quandles are quasigroups. 

In the medial case, this observation allows us to appeal to the classical Bruck--Murdoch--Toyoda theorem, which states the following: for every medial quasigroup $(Q,\ast)$, there exists an abelian group $A$, a fixed element $z\in A$, and two commuting automorphisms $\phi$ and $\psi$ of $A$ such that $(Q,\ast)$ is isomorphic to the medial quasigroup $(A,\cdot)$ defined by
\begin{equation}\label{eq:BMT}
    x\cdot y\coloneq\phi(x)+\psi(y)+z.
\end{equation}

\begin{prop}[{\cite[][Ex.\ 2.2 and Cor.\ 3.4]{jedlicka}}]\label{prop:every-ML}
    Every medial Latin quandle is isomorphic to an Alexander quandle $\Alex(A,\phi)$ such that the map $\id-\phi$ is an automorphism of $A$.
\end{prop}

\begin{proof}
    By the above discussion, every medial Latin quandle is isomorphic to a quasigroup of the form $(A,\cdot)$ defined by \eqref{eq:BMT}. By assumption, $(A,\cdot)$ is idempotent, so taking $x\coloneq 0$ and $y\coloneq 0$ in \eqref{eq:BMT} shows that $z=0$. Therefore, idempotence forces $\phi+\psi=\id$; this recovers \eqref{eq:alex}. Hence, $(A,\cdot)=\Alex(A,\phi)$, as desired. Since $\Alex(A,\phi)$ is Latin, Example \ref{ex:alex} completes the proof.
\end{proof}

\begin{cor}[{\cite{bauer}, \cite[][Prop.\ 12.7]{moufang}}]\label{cor:every-MC}
    Every medial commutative quandle is isomorphic to a midpoint quandle.
\end{cor}

\begin{proof}
    By Proposition \ref{prop:every-ML}, it suffices to show that every commutative Alexander quandle $\Alex(A,\phi)$ is a midpoint quandle. Indeed, commutativity implies that
    \[
    \phi(x)=x\ast 0=0\ast x=x-\phi(x)
    \]
    for all $x\in A$; that is, $2\phi=\id$. Since $\phi$ is invertible, it follows that multiplication by $2$ is invertible. Hence, $A$ is a module over $\D$, and $\phi$ is multiplication by $1/2$, so $\Alex(A,\phi)=A\avg$.
\end{proof}

\begin{rmk}\label{rmk:midpoint}
    The construction of Bauer in \cite{bauer} can be viewed as a way to recover the midpoint quandle corresponding to a given medial commutative quandle under Corollary \ref{cor:every-MC}. Indeed, one can show that the algebraic structures in Bauer's construction (namely, medial idempotent commutative quasigroups, which Bauer calls \emph{quasigroup midpoint algebras}) are the same as medial commutative quandles. This is the reason we chose the name ``midpoint quandles'' for quandles of the form $M\avg$.
    
    Incidentally, midpoint algebras have certain connections to (affine) modules over $\D$; see \cite{bauer,freyd}. By the above discussion, this suggests medial commutative quandles should also be related to affine modules over $\D$; we formalize this idea using an equivalence of categories in Corollary \ref{cor:main}.
\end{rmk}

\section{Quandles versus affine modules}\label{sec:5}
In this section, we categorify Proposition \ref{prop:every-ML} and Corollary \ref{cor:every-MC}. This provides an alternative, module-theoretic perspective on the theory of medial Latin quandles; cf.\ Remark \ref{rmk:midpoint}.

\subsection{Medial Latin quandles}
Henceforth, let $\Med$ denote the category of medial Latin quandles and quandle homomorphisms, and recall that $\L$ denotes the ring of integral Laurent polynomials in $t$ and $1-t$. We show that $\Med$ and $\AM_\L$ are equivalent.

\subsubsection{}
As noted in \cite{burrows,nelson,jedlicka}, the data of a nonempty object \(M\) in $\AM_\L$ is equivalent to the data of an abelian group automorphism $\phi\in\Aut(M)$ such that the map $\id-\phi$ is invertible.\footnote{More explicitly, the correspondence is given by \(t\pinv x \leftrightarrow \varphi\pinv(x)\) and \((1-t)\pinv x\leftrightarrow (\id-\phi)\pinv(x)\) for all \(x\in M\).} So, define a functor \(\Alex\colon \AM_\L\to\Med\) on objects by sending \(M\) to the induced Alexander quandle \(\Alex(M,\phi)\). Let \(\Alex\) fix all morphisms as set-theoretic maps. We begin with two auxiliary results.

\begin{prop}\label{lemma:alex-ftr}
    The assignment \(\Alex\colon \AM_\L\to\Med\) is a functor.
\end{prop}

\begin{proof}
    We only have to show that every affine transformation of $\L$-modules \(f\colon M\to N\) is also a quandle homomorphism \(f\colon \Alex(M,\phi)\to\Alex(N,\psi)\). If \(M\) is empty, then the claim is trivial. Otherwise, \(f\) has the form \(f=T+z\) for some \(\L\)-linear map \(T\colon M\to N\) and some constant \(z\in N\). In particular, \(T\circ\phi = \psi\circ T\), so
    \begin{align*}
    f(x\ast y)&= f(\phi(x-y)+y) \\
   &= (T\circ\phi)(x-y)+T(y)+z\\
   &= (\psi\circ T)(x-y) + f(y)\\
   &= \psi(f(x)-f(y)) + f(y)\\
   &= f(x)\ast f(y)
    \end{align*}
    for all \(x,y\in M\). Hence, \(f\) is a quandle homomorphism.
\end{proof}

\begin{lemma}\label{lemma:linear}
    Let $M$ and $N$ be $\L$-modules. Then every $\Z[t]$-linear map $T\colon M\to N$ is $\L$-linear.
\end{lemma}

\begin{proof}
    Left to the reader; use the fact that $T$ commutes with $t$ and $1-t$ and the fact that $t$ and $1-t$ are invertible.
\end{proof}

\subsubsection{}
We prove the main result of this section.

\begin{thm}\label{thm:main}
    The functor $\Alex\colon \AM_\L\to\Med$ is an equivalence of categories.
\end{thm}

\begin{proof}
    By Propositions \ref{prop:every-ML} and \ref{lemma:alex-ftr} and the discussion preceding Proposition \ref{lemma:alex-ftr}, \(\Alex\) is an essentially surjective functor. Clearly, $\Alex$ is also faithful. Therefore, it suffices to show that \(\Alex\) is full. To that end, let \(f\colon \Alex(M,\phi)\to\Alex(N,\psi)\) be a homomorphism of Latin Alexander quandles, and view \(M\) and \(N\) as \(\L\)-modules as in the discussion preceding Proposition \ref{lemma:alex-ftr}. 
    
    We have to show that $f$ is a morphism in $\AM_\L$. Let \(z\coloneq f(0)\), and define \(T\colon M\to N\) by \(T\coloneq f-z\). Since \(f=T+z\), we only have to show that \(T\) is \(\L\)-linear. By Lemma \ref{lemma:linear}, it suffices to show that \(T\) is a homomorphism of abelian groups such that \(T\circ\phi=\psi\circ T\). Clearly, \(T(0)=0\).

Since \(f\) is a quandle homomorphism, the reader can verify that \(T\) is also a quandle homomorphism. Equivalently,
\begin{equation}\label{eq:T}
    T(\phi(x-y)+y)=\psi(T(x))+(\id-\psi)(T(y))
\end{equation} for all \(x,y\in M\). In particular, given \(m\in M\), take \((x,y)\coloneq(\phi^{-1}(m),0)\) in \eqref{eq:T} to deduce that \(T=\psi\circ T\circ\phi^{-1}\). Therefore, \(T\circ\phi=\psi\circ T\), as desired.

It remains to show that $T$ distributes over addition. To that end, define \(g\colon M^2\to N\) by \[g(x,y)\coloneq T(x+y)-T(x)-T(y).\] We have to show that \(g\equiv 0\). Since \(\phi\) is an endomorphism of abelian groups and \(T\) is a quandle homomorphism, we first compute
    \begin{align*}
    T((a\ast b)+(c\ast d)) &= T(\phi(a)+(\id-\phi)(b)+\phi(c)+(\id-\phi)(d)) \\
    &= T(\phi(a+c)+(\id-\phi)(b+d)) \\
   &= T((a+c)\ast(b+d)) \\
   &= T(a+c)\ast T(b+d)
    \end{align*}
for all \(a,b,c,d\in M\). Since \(T\) is a quandle homomorphism and \(\psi\) is an endomorphism of abelian groups, it follows that
    \begin{align*}
    g((a\ast b),(c\ast d))&= T((a\ast b)+(c\ast d)) - T(a\ast b) - T(c\ast d) \\
   &= [T(a+c)\ast T(b+d)] - [T(a)\ast T(b)] - [T(c)\ast T(d)] \\
   &= \psi(T(a+c)-T(a)-T(c))+(\id-\psi)(T(b\ast d)-T(b)-T(d)) \\
   &= g(a,c)\ast g(b,d)
    \end{align*}
for all \(a,b,c,d\in M\). 

In particular, given \(x,y\in M\), let \[a\coloneq\phi^{-1}(x),\quad b\coloneq 0,\quad c\coloneq0,\quad d\coloneq(\id-\phi)^{-1}(y).\] Then \(x=a\ast b\) and \( y=c\ast d\), so the previous calculation yields \[g(x,y)=g((a\ast b),(c\ast d))=g(a,c)\ast g(b,d)=g(a,0) \ast g(0,d)=0\ast 0=0\] because \(g(-,0)\equiv 0\) and \(g(0,-)\equiv 0\). Hence, \(T\) is \(\L\)-linear, so $f$ is a morphism in $\AM_\L$.
\end{proof}

\subsection{Medial commutative quandles}
Let $\Comm$ denote the category of medial commutative quandles and quandle homomorphisms, and recall that $\D=\Z[1/2]$ denotes the ring of dyadic rationals. Motivated by the study of $\Comm$ in \cite{even}, we show that $\Comm$ and $\AM_\D$ are equivalent.

\begin{cor}\label{cor:main}
    The functor $\operatorname{mid}\colon \AM_\D\to\Comm$ defined by $M\mapsto M\avg$ is an equivalence of categories.
\end{cor}

\begin{proof}
    By Lemma \ref{lemma:res}, the surjective ring homomorphism
    \[
    \L\surj\L/(2t-1)\bij \D
    \]
    induces a fully faithful embedding $\operatorname{res}\colon \AM_\D\inj\AM_\L$. By  Theorem \ref{thm:main} and Remark \ref{rmk:medial}, the composite functor $\Alex\circ \operatorname{res}$ is an equivalence of categories from $\AM_\D$ to the category of midpoint quandles and quandle homomorphisms. By Corollary \ref{cor:every-MC}, the inclusion functor $\iota$ from the latter category into $\Comm$ is also an equivalence of categories. On the other hand, Remark \ref{rmk:medial} shows that $\operatorname{mid}\colon\AM_\D\to\Comm$ is the composite functor
    \[
    \operatorname{mid}=\iota\circ \Alex\circ \operatorname{res}.
    \]
    Hence, $\operatorname{mid}$ is an equivalence of categories.
\end{proof}

\begin{rmk}
    It is interesting to compare Corollary \ref{cor:main} with a certain theorem of Gr\o sfjeld \cite[][Thm.\ 5.2]{rack-roll}. Consider the category $\mathsf{Ab}(\mathsf{Qnd})$ of abelian group objects in the category of quandles, and let $\mathcal{C}$ be the full subcategory of commutative quandles in $\mathsf{Ab}(\mathsf{Qnd})$. One part of Gr\o sfjeld's result can be rephrased to state that the category of modules over $\D$ is equivalent to $\mathcal{C}$. Since all objects in $\mathsf{Ab}(\mathsf{Qnd})$ are also medial quandles (see \cite[][Lem.\ 4.3]{rack-roll}), Corollary \ref{cor:main} can be viewed as a basepoint-free or identity-free version of \cite[][Thm.\ 5.2]{rack-roll} in the sense that morphisms in $\AM_\D$ need not preserve identity elements.

    Another part of Gr\o sfjeld's result can be rephrased to state that the full subcategory of cocommutative quandles in $\mathcal{C}$ is equivalent to the category of vector spaces over $\Z/3$. This makes for an interesting comparison with Corollary \ref{cor:cocomm}; indeed, it is possible to show in purely quandle-theoretic terms that the order of every finite commutative kei is a power of $3$ (see \cite[][Thm.\ 5]{kei}).
\end{rmk}

\subsubsection{}
In particular, combining Theorem \ref{thm:main} and Corollary \ref{cor:main} with Lemmas \ref{lemma:ess-inj} and \ref{lemma:linear} yields the following.

\begin{cor}\label{cor:isom}
Two Latin Alexander quandles $\Alex(M,\phi)$ and $\Alex(N,\psi)$ are isomorphic if and only if $M$ and $N$ are isomorphic as $\L$-modules (or, equivalently, as $\Z[t]$-modules). In particular, two midpoint quandles are isomorphic if and only if they are isomorphic as modules over $\D$.
\end{cor}

\begin{rmk}
    In 2003, Nelson \cite[][Thm.\ 2.1]{nelson} proved that two finite (not necessarily Latin) Alexander quandles \(\Alex(M,\phi)\) and \(\Alex(N,\psi)\) are isomorphic if and only if \((1-t)M\) and \((1-t)N\) are isomorphic as \(\mathbb{Z}[t\pinv]\)-modules. Since multiplication by \(1-t\) is invertible for Latin Alexander quandles, Corollary \ref{cor:isom} shows that Nelson's result still holds if the word ``finite'' is replaced with ``Latin.''
\end{rmk}

\begin{rmk}
    Corollary \ref{cor:isom} fails if we drop the assumption that \(\Alex(M,\phi)\) and \(\Alex(N,\psi)\) are Latin; by Example \ref{ex:alex}, this is equivalent to dropping the assumption that \(1-t\) is invertible in the ground ring. For example, Nelson \cite{nelson} showed that \(\Z[t\pinv]/(9,t-4)\) and \(\Z[t\pinv]/(9,t-7)\) are isomorphic as Alexander quandles but not as \(\Z[t\pinv]\)-modules. Accordingly, these are not \(\L\)-modules.
\end{rmk}

\section{Free objects}\label{sec:6}
As an application of Theorem \ref{thm:main} and Corollary \ref{cor:main}, we completely describe free medial Latin quandles and free medial commutative quandles. As another application, we obtain a structure theorem for finitely generated medial commutative quandles. This solves Question 7.3 and resolves Question 7.1 in \cite{BE}.

\subsection{Free affine modules}

First, we describe free affine modules over a ring $R$. Since affine modules over $R$ form a variety in the sense of universal algebra (see Remark \ref{rmk:lawvere}), free objects in $\AM_R$ exist and satisfy the usual universal property. 

Namely, the free affine $R$-module generated by a set $X$ is an object $\FAM_R(X)\in \AM_R$ along with an injective map \(\iota\colon X\inj \FAM_R(X)\) satisfying the following universal property: for all affine \(R\)-modules \(M\) and all set-theoretic functions \(f\colon X\to M\), there exists a unique affine transformation \(\Tilde{f}\colon \FAM_R(X)\to M\) such that \(\Tilde{f}\circ\iota=f\).

\subsubsection{}
More explicitly, we can construct $\FAM_R(X)$ as follows. If \(X\) is empty, then \(\FAM_R(X)\) is also empty. Otherwise, fix a basepoint \(z\in X\), and let \(\FAM_R(X)\coloneq R^{(X\setminus\{z\})}\) be the free left \(R\)-module generated by the set \(X\setminus\{z\}\). Define \(\iota\colon X\inj \FAM_R(X)\) by \(z\mapsto 0\) and \(x\mapsto x\) for all \(x\in X\setminus\{z\}\). The reader can verify that \(\FAM_R(X)\) satisfies the desired universal property with \(\Tilde{f}\coloneq T+z\), where \(T\colon \FAM_R(X)\to M\) is the unique $R$-linear map obtained by linearly extending the assignment \(x\mapsto f(x)-z\) for all \(x\in X\).

In particular, if \(X\) is a set of cardinality \(1\leq n<\infty\), then \(\FAM_R(X)\cong R^{n-1}\).

\subsection{Free medial Latin quandles}
Since equivalences of categories preserve universal properties, Theorem \ref{thm:main} and Corollary \ref{cor:main} allow us to construct all free objects in $\Med$ and $\Comm$.

Given a set $X$, we construct the free medial Latin quandle \(\FML(X)\) and the free medial commutative quandle $\FMC(X)$ generated by $X$ as follows. If $X$ is empty, then these quandles are also empty. Otherwise, fix a basepoint $z\in X$, and consider the associated free affine modules $\FAM_\L(X)$ and $\FAM_\D(X)$. We define
\[
\FML(X)\coloneq \Alex(\FAM_\L(X),\phi),\qquad \FMC(X)\coloneq \FAM_\D(X)\avg,
\]
where $\phi$ denotes multiplication by $t$. By the above discussion, $\FML(X)$ and $\FMC(X)$ satisfy the desired universal properties. 

In particular, since $\FAM_R(X)\cong R^{n-1}$ for all sets $X$ of cardinality $1\leq n<\infty$, we obtain the following description of nonempty finitely generated free objects in $\Med$ and $\Comm$.

\begin{prop}\label{prop:fml}
    Let $X$ be a set of cardinality \(1\leq n<\infty\). Then the free medial Latin quandle and the free medial commutative quandle generated by $X$ are given by
    \[
    \FML(X)\cong \Alex(\L^{n-1},\phi),\qquad\FMC(X)\cong (\D^{n-1})\avg,
    \]
    where $\phi$ denotes multiplication by $t$.
\end{prop}

\subsubsection{}
In the following corollary, the $n=2$ case provides a positive answer to \cite[][Question 7.3]{BE}.

\begin{cor}[{cf.\ \cite[][Question 7.3]{BE}}]\label{cor:free-comm}
    Let $X$ be a set of cardinality $1\leq n\leq 3$. Then the free commutative quandle generated by $X$ is isomorphic to the free medial commutative quandle $\FMC(X)\cong(\D^{n-1})\avg$ generated by $X$.
\end{cor}

\begin{proof}
    By Proposition \ref{prop:fml}, it suffices to show that the free commutative quandle generated by $X$ is medial. Since \(1\leq n\leq 3\), this follows from \cite[][Prop.\ 3.2]{kepka}, which applies by Lemma \ref{lem:comm-latin}.
\end{proof}

\begin{rmk}
    This argument fails if $n\geq 4$ due to its reliance on \cite[][Prop.\ 3.2]{kepka}. In this case, we suspect that the the non-medial commutative quandles $D(1)$ and $D(3)$ from \cite{moufang} are quotients of the free commutative quandle generated by $X$. Since the former quandles are non-medial, this would imply that the latter quandle is also non-medial, thus making the claim fail for $n\geq 4$.
\end{rmk}

\subsection{Medial commutative rack decomposition}

As one last application of Corollary \ref{cor:main}, we obtain a structure theorem for finitely generated commutative medial quandles. This completely characterizes quandles for which the first part of \cite[][Question 7.1]{BE} has a positive answer.

Recall that $C_{2n+1}=(\Z/(2n+1))\avg$ denotes the cyclic midpoint quandle of order $2n+1$.

\begin{thm}[{cf.\ \cite[][Question 7.1]{BE}}]\label{thm:structure}
    Let \(X\) be a nonempty finitely generated rack. Then \(X\) is medial and commutative if and only if there exists a rack isomorphism \[X\cong (\D^r)\avg \oplus\bigoplus^n_{i=1} C_{2m_i+1}\] for some nonnegative integers \(r,n,m_1,\dots,m_n\geq 0\). 
    
    In particular, every nonempty finite medial commutative rack is isomorphic to a direct sum of cyclic midpoint quandles $C_{2m_i+1}$.
\end{thm}

\begin{proof}
    Note that $\D$ is a PID, and all finite modules over $\D$ are abelian groups of odd order. Therefore, the claim follows after combining the structure theorem for finitely generated modules over a PID with Lemma \ref{lem:comm-latin}, Corollaries \ref{cor:main} and \ref{cor:isom}, and Proposition \ref{prop:fml}.
\end{proof}

\begin{Backmatter}

\printaddress

\end{Backmatter}

\end{document}